\documentclass[11pt]{amsart}

\usepackage[english]{babel}
\usepackage[T1]{fontenc}
\usepackage[utf8]{inputenc}
\usepackage{amsmath}
\usepackage{amssymb}
\usepackage{amsthm}
\usepackage{color}
\usepackage[a4paper]{geometry}
\usepackage{tikz}
\usepackage{mathdots}
\usepackage[all]{xy}
\usepackage{hyperref}
\usepackage{accents}
\usepackage{lmodern}
\usepackage{mathrsfs}

\theoremstyle{definition}
\newtheorem{definition}{Definition}[section]

\newtheorem{rmk}[definition]{Remark}

\theoremstyle{plain}
\newtheorem{theorem}[definition]{Theorem}
\newtheorem{prop}[definition]{Proposition}

\newtheorem{lem}[definition]{Lemma}

\newcommand{\mb}{\mathbb}
\newcommand{\mc}{\mathcal}

\newcommand{\bs}{\boldsymbol}

\usepackage{scalerel,stackengine}
\stackMath
\newcommand\reallywidehat[1]{%
\savestack{\tmpbox}{\stretchto{%
  \scaleto{%
    \scalerel*[\widthof{\ensuremath{#1}}]{\kern-.6pt\bigwedge\kern-.6pt}%
    {\rule[-\textheight/2]{1ex}{\textheight}}
  }{\textheight}%
}{0.5ex}}%
\stackon[1pt]{#1}{\tmpbox}%
}
\begin{document}

\title[A Remark on the Set of Exactly Approximable Vectors]{A Remark on the Set of Exactly Approximable Vectors in the Simultaneous Case}

\thanks{The author was supported by SNSF grant 200021--182089.}

\author{Reynold Fregoli}

\address{Department of Mathematics\\ 
University of Z\"urich\\ 
Winterthurerstrasse 190, 8057\\ 
Z\"urich}
\email{reynoldfregoli@gmail.com}

\subjclass{11J13, 11J83}

\begin{abstract}We compute the Hausdorff dimension of the set of $\psi$-exactly approximable vectors, in the simultaneous case, in dimension strictly larger than $2$ and for approximating functions $\psi$ with order at infinity less than or equal to $-2$. Our method relies on the analogous result in dimension $1$, proved by Yann Bugeaud and Carlos Moreira, and a version of Jarn\'ik's Theorem on fibres.
\end{abstract}

\maketitle

\section{Introduction and Main Result}

Let $\psi:\mb N\to \mb (0,1]$ be a non-increasing function. The set of $\psi$-well-approximable numbers on the interval $[0,1]$ is defined as
$$W(1,\psi):=\left\{\alpha\in[0,1]:\left|\alpha-\frac{p}{q}\right|<\psi(q)\mbox{ for infinitely many }p,q\in\mb N\right\}.$$
This set is central to the theory of Diophantine approximation and, throughout the last century, multiple authors have contributed to the endeavour of establishing its metrical properties. In 1924, Khintchine \cite{Khi24} proved that its Lebesgue measure can only attain the values $0$ or $1$, depending on the convergence of the series $\sum_{q\in\mb N}q\cdot\psi(q)$. A few years later, Jarn\'ik \cite{Jar29} and, independently, Besicovitch \cite{Bes34} showed that the Hausdorff dimension of the set $W(1,\psi)$, in the special case $\psi(q)=q^{-\nu}$, with $\nu>2$, is equal to $2/\nu$. Finally, in 1931, Jarn\'ik \cite{Jar31} computed the $f$-Hausdorff measure of $W(1,\psi)$, under mild assumptions on the functions $\psi$ and $f$, thus producing a complete measure-theoretic description of this set. We refer the reader to \cite{BDV06} for a more precise formulation of these statements and additional references.

In \cite{Jar31}, Jarn\'ik introduced the following subset of $W(1,\psi)$:
$$E(1,\psi):=W(1,\psi)\setminus\bigcup_{0<c<1}W(1,c\psi).$$
The numbers belonging to $E(1,\psi)$ are known as $\psi$-\emph{exactly}-approximable, as the upper bound on their rational approximations cannot be improved. Relying on the theory of continued fractions, Jarn\'ik could show that the set $E(1,\psi)$ is uncountable. However, the true size of this set was determined only quite recently by Yann Bugeaud \cite[Theorem 1]{Bug03}, and, in greater generality, by Yann Bugeaud and Carlos Moreira \cite[Theorem 4]{BM11}. Both works also rely on the theory of continued fractions. Their results can be summarised as follows.
\begin{theorem}[Bugeaud-Moreira]
\label{prop:Bugeaud}
Assume that the product $q^{2}\cdot\psi(q)$ tends to $0$ as $q$ approaches infinity. Then, one has that
$$\textup{dim} E(1,\psi)=\textup{dim} W(1,\psi)=\frac{2}{\lambda},$$
where $\lambda$ is the lower order at infinity\footnote{The lower order at infinity of a function $g:\mb N\to (0,+\infty)$ is defined as
$$\liminf_{x\to \infty}\frac{\log g(x)}{\log x}.$$ The upper order at infinity of the function $g$ is defined analogously but with a limit superior in place of the limit inferior.} of the function $1/\psi$ and $\dim$ denotes the Hausdorff dimension.
\end{theorem}
The reader may also refer to \cite{FW22} for recent results on the Fourier dimension of the set $E(1,\psi)$.

The problem of studying $\psi$-approximable numbers $\alpha\in\mb R$ may naturally be generalised to vectors $\bs\alpha\in\mb R^n$ with $n\geq 2$. In this setting, two types of approximation have been widely studied: namely, the simultaneous and the dual type. The set of $\psi$-simultaneously-well-approximable vectors is defined as     
$$W(n,\psi):=\left\{\bs\alpha\in[0,1]^n: \max_{i=1}^{n}\left|\alpha_i-\frac{p_i}{q}\right|<\psi(q)\mbox{ for infinitely many }q\in\mb N,\ \bs p\in \mb Z^n\right\},$$
while its dual counterpart, where the vector $\bs\alpha$ is interpreted as a linear form, is defined as
\begin{multline}
W^{*}(n,\psi):=\Bigg\{\bs\alpha\in[0,1]^n: \left|\alpha_1 q_1+\dotsb \alpha_n q_n-p\right|<\psi\left(\max_{i=1}^{n}|q_i|\right) \\
\mbox{ for infinitely many }\bs q\in\mb Z^n\setminus\{\bs 0\},p\in \mb Z\Bigg\}.\nonumber
\end{multline}
A generalisation of Khintchine's measure-theoretic result, known as the Khintchine-Groshev Theorem, is well-known to hold in both settings \cite{BV10}. Jarn\'ik's dimensional result for the set $W(1,\psi)$ has also been generalised to both the simultaneous and the dual cases (see \cite{BD86} and references therein). Conversely, the set of $\psi$-exactly approximable numbers has so far eluded any such generalisation. A first contribution towards this was made in \cite{BDV01}, where Beresnevich, Dickinson, and Velani study metric properties of the set $W(n,\psi)\setminus W(n,\varphi)$ (and its dual) for two given approximating functions $\psi$ and $\varphi$. Their techniques, however, rest on the assumption that the ratio $\psi/\varphi$ tends to infinity \cite[Section 2.5]{BDV01} and hence, do not apply to the setting of exact approximation in the sense of Jarn\'ik. In a more general context, the same problem was studied by Bandi, Ghosh, and Nandi. In \cite{BGN23}, these authors consider the set of $\psi$-exactly approximable points in a general metric space, with respect to a discrete "well-distributed" subset, and prove related dimensional estimates. Nonetheless, the set of rational points in $\mb R^n$ is not well-distributed in the sense of \cite{BGN23}, leaving the problem of generalising Bugeaud's and Moreira's result to dimension $n\geq 2$ wide open.

Let
$$E(n,\psi):=W(n,\psi)\setminus\bigcup_{0<c<1}W(n,c\psi).$$
The purpose of this short note is to establish the Hausdorff dimension of the set $E(n,\psi)$ for $n\geq 3$ and under some assumptions on the decay of the function $\psi$. Our main result reads as follows.

\begin{theorem}
\label{prop:mainres}
Let $\psi:\mb N\to \mb (0,1]$ be as in Theorem \ref{prop:Bugeaud}, and assume that upper and lower order at infinity of the function $1/\psi$ coincide and are equal to $\lambda\geq 2$. Then, for $n\geq 3$ one has that
$$\dim E(n,\psi)=\dim W(n,\psi)=\frac{n+1}{\lambda},$$
where $\dim$ denotes the Hausdorff dimension.
\end{theorem}

Our proof relies on the simple observation that, in the simultaneous case, a vector $\bs\alpha\in\mb R^n$ with first component $\alpha_1$ not belonging to $W(1,c\psi)$ cannot lie in the set $W(n, c\psi)$. In view of this, Theorem \ref{prop:Bugeaud}, in combination with a "fibred" version of Jarn\'ik's theorem, is enough to yield the desired dimensional estimate. This approach heavily depends on the hypothesis that the order at infinity of the function $1/\psi$ is larger than or equal to $2$, as this is required to apply Theorem \ref{prop:Bugeaud}. For technical reasons, we are also unable to deal with the case $n=2$, as our version of Jarn\'ik's Theorem rests on Gallagher's extension of Khintchine's Theorem \cite[Theorem 1]{Gal65}, which is only valid in dimension greater or equal to $2$. The analogue of our result in the dual case (for $n\geq 2$) also appears in no way attainable through the methods that we develop here.

\begin{rmk}
\label{rmk:1}
As a consequence of our approach, we are also able to show that, when $\psi(q)=\psi_{\nu}(q)=q^{-\nu}$ with $\nu>2$, the following implication holds:
$$\mc{H}^{2/\nu}(E(1,\psi_{\nu}))>0\Rightarrow \mc{H}^{(n+1)/\nu}(W(n,\psi_{\nu}))=\infty.$$
However, the $(2/\nu)$-Hausdorff measure of the set $E(1,\psi_{\nu})$ is at present unknown (see comments following Theorem 2 in \cite{Bug03}). This implication will be proved at the end of Section \ref{sec:proof}.
\end{rmk}

While working on the present note, the author became aware of the work in preparation of H. Koivusalo, J. Levesley, B. Ward, and X. Zhang on a similar problem, which, to the best of his knowledge, explores entirely different techniques and does not supersede the result presented here.

\subsection*{Acknowledgements}
The author is indebted to Victor Beresnevich and Alexander Gorodnik for many fruitful discussions, and to Ben Ward for comments on an early version of the manuscript.

\section{Reduction to Fibres}

For $\alpha\in\mb R$, $n\geq 1$, and $\psi:\mb N\to \mb (0,1]$ we define
$$W(\alpha,n,\psi):=\{\bs x\in[0,1]^n:(\alpha, \bs x)\in W(n+1,\psi)\}.$$

Then, the following lemma holds.

\begin{lem}
\label{lem:1}
For any $\alpha\in E(1,\psi)$ and $\bs x\in W(\alpha,n,\psi)$ it holds $(\alpha, \bs x)\in E(n+1,\psi)$.
\end{lem}

\begin{proof}
Since the number $\alpha$ cannot be approximated by the function $c\psi$ for any $0<c<1$, the vector $(\alpha, \bs x)$ cannot be simultaneously approximated by any such function. It follows that $(\alpha, \bs x)\notin W(n+1,c\psi)$ for all $0<c<1$. On the other hand, by definition of $W(\alpha,n,\psi)$, one has that $(\alpha, \bs x)\in W(n+1,\psi)$.
\end{proof}

In view of Lemma \ref{lem:1}, the following holds
\begin{equation}
\label{eq:containment}
\bigcup_{\alpha\in E(1,\psi)}\{\alpha\}\times W(\alpha,n,\psi)\subset E(n+1,\psi)\subset W(n+1,\psi).
\end{equation}

By (\ref{eq:containment}) and Marstrand's Slicing Lemma \cite[Corollary 7.12]{Fal14}, we deduce that
\begin{equation}
\label{eq:1}
\dim E(1,\psi)+\inf_{\alpha\in E(1,\psi)}\dim W(\alpha,n,\psi)\leq \dim E(n+1,\psi)\leq \dim W(n+1,\psi).\nonumber
\end{equation}

Moreover, by Theorem \ref{prop:Bugeaud}, the equality $\dim E(1,\psi)=\dim W(1,\psi)$ holds. Hence, it follows from (\ref{eq:1}) that
\begin{equation}
\label{eq:2}
\dim W(1,\psi)+\inf_{\alpha\in E(1,\psi)}\dim W(\alpha,n,\psi)\leq \dim E(n+1,\psi)
\leq \dim W(n+1,\psi).
\end{equation}

We will show in Section \ref{sec:proof} that the left-hand side and the right-hand side of the above chain of inequalities coincide.

\section{Jarn\'ik on Fibres}

Let $\psi:\mb N\to \mb (0,1]$ and let $n=k+l$ with $k,l\geq 1$. For $\bs\alpha\in\mb R^l$, we define the set
$$W(\bs\alpha,k,\psi):=\left\{\bs x\in[0,1]^k: (\bs\alpha,\bs x)\in W(n,\psi)\right\}.$$

The goal of this section is to prove the following result, which in some sense may be regarded as a version of Jarn\'ik's Theorem on coordinate subspaces (compare with \cite[Theorem 3.4]{BRV16}).

\begin{prop}
\label{prop:JoF}
Let $\bs\alpha\in W(l,\psi)$ and let $0< s\leq k$, with $k\geq 2$. If  $$\liminf_{q\to \infty}q^{k}\cdot \psi(q)^{s}>0,$$ then, one has that
$$\mathcal{H}^{s}(W(\bs\alpha,k,\psi))=\mc{H}^{s}\left([0,1]^k\right),$$
where $\mathcal H^s$ denotes the $s$-Hausdorff measure.
\end{prop}

To prove Proposition \ref{prop:JoF} we follow \cite{RSS17}. Before proceeding to the proof, we need a slight change of language. For $\phi:\mb N\to \mb R$ we introduce the sets
$$\tilde W(n,\phi):=\left\{\bs x\in[0,1]^n: \max_{i=1}^{n}\|q x_i\|<\phi(q)\mbox{ for infinitely many }q\in\mb N\right\}$$
and 
$$\tilde W(\bs\alpha,k,\phi):=\left\{\bs x\in[0,1]^k: (\bs\alpha,\bs x)\in\tilde  W(n,\phi)\right\}.$$
Henceforth, we will assume that $\phi(q):=q\cdot\psi(q)$. With this notation, we have that
$$\tilde W(\bs\alpha,k,\phi)=W(\bs\alpha,k,\psi).$$
Moreover, the assumption of Proposition \ref{prop:JoF} now reads as $\liminf_{q\to \infty}q^{k-s}\phi(q)^{s}>0$. 

\begin{lem}
\label{lem:cct}
Under the hypotheses of Proposition \ref{prop:JoF} it holds that
$$\sum_{\substack{q\ \textup{s.t}\\ \max_i\|q\alpha_i\|<\phi(q)}}\min\left\{q^{k-s}\phi(q)^{s},1/2\right\}=+\infty.$$
\end{lem}

\begin{proof}
Since $\bs\alpha\in W(l,\psi)$, the condition $\max_i\|q\alpha_i\|<\phi(q)$ must be satisfied for infinitely many values of $q$. Then, the fact that $\liminf_{q\to\infty } q^{k-s}\phi(q)^s>0$ implies the divergence of the series.
\end{proof}

For $\bs \alpha\in W(l,\psi)$ let $$\phi_{\bs\alpha}(q):=
\begin{cases}
\phi(q) \quad\mbox{if }\max_{i}\|q\alpha_i\|<\phi(q) \\
0\quad\mbox{otherwise}
\end{cases}.$$

By Gallagher's extension of Khintchine's Theorem\footnote{Note that in \cite[Theorem 1]{Gal65} the approximating function is assumed less than $1$. For this reason we consider the function $\min\left\{q^{k-s}\phi(q)^{s},1/2\right\}$ in Lemma \ref{lem:cct}.} \cite[Theorem 1]{Gal65} and Lemma \ref{lem:cct}, we have that, under the hypotheses of Proposition \ref{prop:JoF}, the following equality holds:
$$\mc H^{k}\left(\tilde W(k,q\mapsto q^{1-s/k}\cdot \phi_{\bs\alpha}(q)^{s/k})\right)=1.$$
Moreover, setting
$$\psi_{\bs\alpha}(q):=
\begin{cases}
\psi(q) \quad\mbox{if }\max_{i}\|q\alpha_i\|<q\cdot \psi(q)=\phi(q) \\
0\quad\mbox{otherwise}
\end{cases},$$
we have that
$$\tilde W(k,q\mapsto q^{1-s/k}\cdot\phi_{\bs\alpha}(q)^{s/k})=\tilde W(k,q\mapsto q\cdot\psi_{\bs\alpha}(q)^{s/k})=W(k,\psi_{\bs\alpha}^{s/k}).$$
Then, by the Mass Transference Principle \cite[Theorem 3.5]{BRV16}, we deduce that whenever
$$\mc H^{k}\left(W(k,\psi_{\bs\alpha}^{s/k})\right)=\mc H^{k}\left([0,1]^k\right),$$
the following must hold:
$$\mc H^s\left(W(k,\psi_{\bs\alpha})\right)=\mc H^{s}\left([0,1]^k\right).$$
Hence, the proof is concluded on noting that $W(k,\psi_{\bs\alpha})=W(\bs\alpha,k,\psi)$.

\section{Proof of Theorem \ref{prop:mainres}}
\label{sec:proof}

Assume that $l=1$ and $k=n-1\geq 2$. By definition of order at infinity, we have that for any $\varepsilon>0$
$$\psi(q)\geq q^{-\lambda-\varepsilon}$$
for all sufficiently large values of $q$.
Then, for $0<s<(n-1)/\lambda$ and $\varepsilon>0$ so small that $(n-1)-s\lambda>s\varepsilon$, we find
\begin{equation}
\label{eq:3}
 \liminf_{q\to \infty}q^{n-1}\cdot \psi(q)^{s}\geq \liminf_{q\to \infty}q^{(n-1)-s(\lambda+\varepsilon)}>0.
\end{equation}
It follows from Proposition \ref{prop:JoF} that $\mc{H}^{s}\left(W(\alpha,n-1,\psi)\right)=+\infty$ for all $s< (n-1)/\lambda$ and all $\alpha\in E(1,\psi)$. This shows that $\dim W(\alpha,n-1,\psi)\geq (n-1)/\lambda$ for all $\alpha\in E(1,\psi)$. Thus, by (\ref{eq:2}), we deduce that
\begin{multline}
\frac{2}{\lambda}+\frac{n-1}{\lambda}=\dim W(1,\psi)+\inf_{\alpha\in E(1,\psi)}\dim W(\alpha,n-1,\psi)\leq \dim E(n,\psi) \\
\leq \dim W(n,\psi)=\frac{n+1}{\lambda}.\nonumber
\end{multline}
For a proof of the last equality, the reader may refer to \cite{Jar31} or \cite[Corollary 1]{BV06} (case $n=1$). This shows that $\dim E(n,\psi)=\dim W(n,\psi)$.

As for Remark \ref{rmk:1}, we observe that when $\psi(q)=\psi_{\nu}(q)=q^{-\nu}$, Equation (\ref{eq:3}) holds also for $s=(n-1)/\nu$. Hence, we have $\mc{H}^{(n-1)/\nu}(W(\alpha,n-1,\psi_{\nu}))=\infty$. The claim follows from \cite[Lemma 4]{BV06}.

\bibliographystyle{alpha}

\end{document}